\documentclass[12pt]{amsart}

\usepackage{amsmath}
\usepackage{amssymb}
\usepackage{amsfonts}
\usepackage{amsthm}
\usepackage{enumerate}
\usepackage{hyperref}
\usepackage{color}

\theoremstyle{plain}
\newtheorem{thm}{Theorem}[section]

\newtheorem{cor}[thm]{Corollary}

\theoremstyle{definition}
\newtheorem{dfn}[thm]{Definition}

\newtheorem{rem}[thm]{Remark}

\newtheorem{dfns-rems}[thm]{Definitions and Remarks}
\newtheorem{notas-rems}[thm]{Notations and Remarks}
\newtheorem{exmps-rems}[thm]{Examples and Remarks}
\newtheorem{cons}[thm]{Construction}

\makeatletter
\@namedef{subjclassname@2020}{%
  \textup{2020} Mathematics Subject Classification}
\makeatother

\DeclareMathOperator{\ind-match}{ind-match}
\DeclareMathOperator{\ord-match}{ord-match}
\DeclareMathOperator{\sord-match}{{\it s}-ord-match}

\DeclareMathOperator{\1ord-match}{{\it 1}-ord-match}
\DeclareMathOperator{\2ord-match}{{\it 2}-ord-match}
\DeclareMathOperator{\3ord-match}{{\it 3}-ord-match}
\DeclareMathOperator{\mord-match}{{\it m}-ord-match}

\begin{document}


\title[depth stability of symbolic powers of cover ideals]{On the index of depth stability of symbolic powers of cover ideals of graphs}


\author[S. A. Seyed Fakhari]{S. A. Seyed Fakhari}

\address{S. A. Seyed Fakhari, Departamento de Matem\'aticas\\Universidad de los Andes\\Bogot\'a\\Colombia.}

\email{s.seyedfakhari@uniandes.edu.co}

\author[S. Yassemi]{S. Yassemi}

\address{S. Yassemi, Department of Mathematics, Purdue University, West Lafayette, IN, USA}

\email{syassemi@purdue.edu}


\dedicatory{Dedicated with gratitude to our friend Ngo Viet Trung \\ on the occasion of his 70th birthday}


\begin{abstract}
Let $G$ be a graph with $n$ vertices and let $S=\mathbb{K}[x_1,\dots,x_n]$ be the
polynomial ring in $n$ variables over a field $\mathbb{K}$. Assume that $I(G)$ and $J(G)$ denote the edge ideal and the cover ideal of $G$, respectively. We provide a combinatorial upper bound for the index of depth stability of symbolic powers of $J(G)$. As a consequence, we compute the depth of symbolic powers of cover ideals of fully clique-whiskered graphs. Meanwhile, we determine a class of graphs $G$ with the property that the Castelnuovo--Mumford regularity of $S/I(G)$ is equal to the induced matching number of $G$.
\end{abstract}






\thanks{}


\maketitle


\section{Introduction} \label{sec1}

Let $\mathbb{K}$ be a field and $S=\mathbb{K}[x_1,\dots,x_n]$ be the
polynomial ring in $n$ variables over the field $\mathbb{K}$, and let
$I\subset S$ be a monomial ideal. Brodmann \cite{b} proves that ${\rm depth}(S/I^k)$ is constant for large $k$. The smallest integer $k\geq 1$ such that ${\rm depth}(S/I^m)=\lim_{t\rightarrow \infty}{\rm depth}(S/I^t)$ for all $m\geq k$ is called the {\it index of depth stability of powers} of $I$ and is denoted by ${\rm dstab}(I)$. It is of great interest to compute the limit of the sequence $\{{\rm depth}(S/I^k)\}_{k=1}^{\infty}$ and to determine or bound its index of stability. It is also natural to consider similar problems for symbolic powers. More precisely, for every integer $k\geq 1$, let $I^{(k)}$ denote the $k$-th symbolic power of $I$. It follows from \cite[Theorem 4.7]{ht2} that the sequence $\{{\rm depth}(S/I^{(k)})\}_{k=1}^{\infty}$ is convergent when $I$ is a squarefree monomial ideal. For any squarefree monomial ideal of $S$, the smallest integer $k\geq 1$ such that ${\rm depth}(S/I^{(m)})=\lim_{t\rightarrow \infty}{\rm depth}(S/I^{(t)})$ for all $m\geq k$ is called the {\it index of depth stability of symbolic powers} of $I$ and is denoted by ${\rm sdstab}(I)$. In this paper, we study the index of depth stability of symbolic powers of cover ideals of graphs.

For any graph $G$, let $J(G)$ denote its cover ideal. The depth of high symbolic powers of cover ideals of graphs has been studied by Hoa et al. \cite{hktt}. In that paper, the authors provide a combinatorial description for the limit value of the depth of symbolic powers of cover ideal of graphs. In fact, they prove something stronger. They show in \cite[Theorem 3.4]{hktt} that for any graph $G$ with $n$ vertices and for every integer $k\geq 2\ord-match(G)-1$, we have$${\rm depth}(S/J(G)^{(k)})=n-\ord-match(G)-1.$$Here, $\ord-match(G)$ denotes the ordered matching number of $G$ (see Definition \ref{om}). In particular, it follows from this result that ${\rm sdstab}(J(G))\leq 2\ord-match(G)-1$. In Theorem \ref{main}, we provide a refinement of this inequality. Indeed, for every graph $G$ and for every integer $s\geq 1$, we define the notion of $s$-ordered matching number of $G$, denoted by $\sord-match(G)$ (see Definition \ref{som}). It is proven in Theorem \ref{main} that if $\sord-match(G)=\ord-match(G)$, for some integer $s\geq 2$, then$${\rm sdstab}(J(G))\leq 2\ord-match(G)-2s+2.$$As a consequence of the above inequality, in Corollary \ref{whisk}, we compute the depth of symbolic powers of cover ideals of fully clique-whiskered graphs. In particular, we will see that for any such graph, we have ${\rm sdstab}(J(G))\leq 2$. For a graph $G$, let $I(G)$ denote its edge ideal. As the next main result, we use the idea of the proof Theorem \ref{main} to determine a class of graphs such that the Castelnuovo--Mumford regularity of $S/I(G)$ is equal to the induced matching number of $G$.

Let $G$ be a bipartite graph. We know from \cite[Corollary 2.6]{grv} that for any integer $k\geq 1$, the equality $J(G)^{(k)}=J(G)^k$ holds. In particular, ${\rm dstab}(J(G))={\rm sdstab}(J(G))$. In this case, Hang and Trung \cite[Theorem 3.6]{ht} strengthen \cite[Theorem 3.4]{hktt} by proving the inequality$${\rm sdstab}(J(G))\leq \ord-match(G).$$ In Theorem \ref{bipart}, we provide an alternative proof for this inequality. Indeed, the proof in \cite{ht} is based on a formula due to Takayama \cite[Theorem 2.2]{t}, while our proof follows from a polarization argument.


\section{Preliminaries} \label{sec1'}

In this section, we provide the definitions and the basic facts which will be used in the next section.

Let $G$ be a graph with vertex set $V(G)=\big\{x_1, \ldots,
x_n\big\}$ and edge set $E(G)$ (by abusing the notation, we identify the vertices of $G$ with the variables of $S$). A subgraph $H$ of $G$ is called {\it induced} provided that two vertices of $H$ are adjacent if and only if they are adjacent in $G$. A subset $W$ of $V(G)$ is called an {\it independent subset} of $G$ if there are no edges among the vertices of $W$. The cardinality of the largest independent subset of $G$ is the {\it independence number} of $G$, and will be denoted by $\alpha(G)$. A subset $C$ of $V(G)$ is called a {\it vertex cover} of the graph $G$ if every edge of $G$ is incident to at least one vertex of $C$. A vertex cover $C$ is called a {\it minimal vertex cover} of $G$ if no proper subset of $C$ is a vertex cover of $G$. A {\it matching} in $G$ is a subgraph consisting of pairwise disjoint edges. If the subgraph is an induced subgraph, the matching is an {\it induced matching}. The cardinality of the maximum induced matching of $G$ is the {\it induced matching number} of $G$ and will be denoted by $\ind-match(G)$.

Next, we define the notion of ordered matching for a graph. It was introduced in \cite{cv} and plays a key role in this paper.

\begin{dfn} \label{om}
Let $G$ be a graph, and let $M=\big\{\{a_i,b_i\}\mid 1\leq i\leq r\big\}$ be a
nonempty matching of $G$. We say that $M$ is an {\it ordered matching} of
$G$ if the following conditions hold:
\begin{itemize}
\item[(1)] $A:=\{a_1,\ldots,a_r\} \subseteq V(G)$ is an
    independent subset of $G$; and

\item[(2)] $\{a_i, b_j\}\in E(G)$ implies that $i\leq j$.
\end{itemize}
The {\it ordered matching number} of $G$, denoted by $\ord-match(G)$, is
defined to be $$\ord-match(G)=\max\{|M|\mid M\subseteq E(G)\ {\rm is\ an\
ordered\ matching\ of} \ G\}.$$
\end{dfn}

In the following definition, we provide a refinement of the notion of ordered matching.

\begin{dfn} \label{som}
Let $G$ be a graph, and let $M=\big\{\{a_i,b_i\}\mid 1\leq i\leq r\big\}$ be a
nonempty matching of $G$. For an integer $s\geq 1$, we say that $M$ is an {\it s-ordered matching} of
$G$ if the following conditions hold:
\begin{itemize}
\item[(1)] $r\geq s$;

\item[(2)] $A:=\{a_1,\ldots,a_r\} \subseteq V(G)$ is an
    independent subset of $G$; and

\item[(3)] $\{a_i, b_j\}\in E(G)$ implies that either $i=j$ or $i\leq j-s$ (in particular, if $r=s$, then $\{a_i, b_j\}\in E(G)$ implies that $i=j$).
\end{itemize}
The {\it s-ordered matching number} of $G$, denoted by $\sord-match(G)$, is
defined to be $$\sord-match(G)=\max\{|M|\mid M\subseteq E(G)\ {\rm is\ an\
{\it s}-ordered\ matching\ of} \ G\}.$$If $G$ has no $s$-ordered matching, then we define its $s$-ordered matching number to be $-\infty$.
\end{dfn}

It follows from the definition that$$\ord-match(G)=\1ord-match(G)\geq \2ord-match(G)\geq \3ord-match(G)\geq \cdots.$$

The edge ideal $I(G)$ of $G$ is the ideal of $S$ generated by the squarefree  monomials  $x_ix_j$, where $\{x_i, x_j\}$ is an edge of $G$. The Alexander dual of the edge ideal of $G$ in $S$, i.e., the
ideal $$J(G)=I(G)^{\vee}=\bigcap_{\{x_i,x_j\}\in E(G)}(x_i,x_j),$$ is called the
{\it cover ideal} of $G$ in $S$. The reason for this name is due to the
well-known fact that the generators of $J(G)$ correspond to minimal vertex covers of $G$. We refer to \cite{hv2} and \cite{s13} for surveys about homological properties of powers of cover ideals.

Let $I$ be an ideal of $S$ and let ${\rm Min}(I)$ denote the set of minimal primes of $I$. For every integer $k\geq 1$, the $k$-th {\it symbolic power} of $I$,
denoted by $I^{(k)}$, is defined to be$$I^{(k)}=\bigcap_{\frak{p}\in {\rm Min}(I)} {\rm Ker}(S\rightarrow (S/I^k)_{\frak{p}}).$$Let $I$ be a squarefree monomial ideal in $S$ and suppose that $I$ has the irredundant
primary decomposition $$I=\frak{p}_1\cap\ldots\cap\frak{p}_r,$$ where every
$\frak{p}_i$ is an ideal of $S$ generated by a subset of the variables of
$S$. It follows from \cite[Proposition 1.4.4]{hh} that for every integer $k\geq 1$, $$I^{(k)}=\frak{p}_1^k\cap\ldots\cap
\frak{p}_r^k.$$In particular, for any graph $G$ and for each integer $k\geq 1$, we have$$J(G)^{(k)}=\bigcap_{\{x_i,x_j\}\in E(G)}(x_i,x_j)^k.$$

Let $M$ be a finitely generated graded $S$-Module. The {\it projective dimension} and the {\it Castelnuovo-Mumford regularity} (or simply, {\it regularity}) of $M$, are defined as follows.
$${\rm pd}(M)=\max\{i|\ {\rm Tor}_i^S(\mathbb{K}, M)\neq0\}, \ \ \ \ {\rm reg}(M)=\max\{j-i|\ {\rm Tor}_i^S(\mathbb{K}, M)_j\neq0\}.$$

Let $I$ be a monomial ideal of
$S$ with minimal generators $u_1,\ldots,u_m$,
where $u_j=\prod_{i=1}^{n}x_i^{a_{i,j}}$, $1\leq j\leq m$. For every $i$
with $1\leq i\leq n$, let $a_i=\max\{a_{i,j}\mid 1\leq j\leq m\}$, and
suppose that $$T=\mathbb{K}[x_{1,1},x_{1,2},\ldots,x_{1,a_1},x_{2,1},
x_{2,2},\ldots,x_{2,a_2},\ldots,x_{n,1},x_{n,2},\ldots,x_{n,a_n}]$$ is a
polynomial ring over the field $\mathbb{K}$. Let $I^{{\rm pol}}$ be the squarefree
monomial ideal of $T$ with minimal generators $u_1^{{\rm pol}},\ldots,u_m^{{\rm pol}}$, where
$u_j^{{\rm pol}}=\prod_{i=1}^{n}\prod_{k=1}^{a_{i,j}}x_{i,k}$, $1\leq j\leq m$. The ideal $I^{{\rm pol}}$
is called the {\it polarization} of $I$. We know from \cite[Corollary 1.6.3]{hh} that polarization preserves the projective dimension, i.e.,$${\rm pd}(S/I)={\rm pd}(T/I^{{\rm pol}}).$$

\begin{cons} [\cite{s3}, Page 101] \label{cons}
Let $G$ be a graph with vertex set $V(G)=\big\{x_1, \ldots,
x_n\big\}$ and edge set $E(G)$. For each integer $k\geq 1$, we define the graph $G_k$ to be the graph with the vertex set $V(G_k)=\{x_{i,p}\mid 1\leq i\leq n,  1\leq p\leq k\}$ and the edge set$$E(G_k)=\big\{\{x_{i,p}, x_{j,q}\}\mid \{x_i, x_j\}\in E(G) \  {\rm and} \  p+q\leq k+1\big\}.$$By \cite[Lemma 3.4]{s3}, we have $(J(G)^{(k)})^{{\rm pol}}=J(G_k)$.
\end{cons}


\section{Main results} \label{sec2}

Let $G$ be a graph with $n$ vertices. As the first main result of this paper, we determine a combinatorial upper bound for the index of depth stability of symbolic powers of cover ideals of $G$. Indeed, we know from \cite[Theorem 3.4]{hktt} that$${\rm sdstab}(J(G))\leq 2\ord-match(G)-1.$$In the following theorem, we provide an improvement of the above inequality.

\begin{thm} \label{main}
Let $G$ be a graph with $n$ vertices and suppose that $s\geq 1$ is the largest integer such that $\sord-match(G)=\ord-match(G)$.
\begin{itemize}
\item[(i)] If $s=1$, then for every integer $k$ with $k\geq 2\ord-match(G)-1$, we have ${\rm depth}(S/J(G)^{(k)})=n-\ord-match(G)-1$.

\item[(ii)] If $s\geq 2$, then for every integer $k$ with $k\geq 2\ord-match(G)-2s+2$, we have ${\rm depth}(S/J(G)^{(k)})=n-\ord-match(G)-1$.
\end{itemize}
\end{thm}

\begin{proof}
(i) is known by \cite[Theorem 3.4]{hktt}. So, we prove (ii).

Without loss of generality suppose that $G$ has no isolated vertex. Combining \cite[Proposition 2.4]{v} and \cite[Theorem 2.8]{cv}, we have$${\rm depth}(S/J(G)^{(k)})\geq n-\ord-match(G)-1,$$for every integer $k\geq 1$. Thus, we prove that$${\rm depth}(S/J(G)^{(k)})\leq n-\ord-match(G)-1,$$for every integer $k\geq 2\ord-match(G)-2s+2$. So, assume that$$k\geq 2\ord-match(G)-2s+2$$is an integer and suppose that $V(G)=\{x_1, \ldots, x_n\}$ is the vertex set of $G$. Using the Auslander-Buchsbaum formula, we need to show that$${\rm pd}(S/J(G)^{(k)})\geq \ord-match(G)+1.$$As polarization preserves the projective dimension, it is enough to prove that$${\rm pd}(T/(J(G)^{(k)})^{{\rm pol}})\geq \ord-match(G)+1,$$where $T$ is a new polynomial ring containing $(J(G)^{(k)})^{{\rm pol}}$. We know from \cite[Lemma 3.4]{s3} that $(J(G)^{(k)})^{{\rm pol}}$ is the cover ideal of the graph $G_k$ introduced in Construction \ref{cons}. Using Terai's theorem \cite[Theorem 8.1.10]{hh}, it is sufficient to show that $${\rm reg}(T/(I(G_k))\geq \ord-match(G).$$ By \cite[Lemma 2.2]{k}, we know that$${\rm reg}(T/(I(G_k))\geq \ind-match(G_k).$$We show that $\ind-match(G_k)\geq \ord-match(G)$ and this completes the proof. Set $t=\ord-match(G)$ and assume without loss of generality that $\big\{\{x_i,x_{t+i}\}\mid 1\leq i\leq t\big\}$ is an ordered matching of $G$. Since $\sord-match(G)=\ord-match(G)$, we may suppose that
\begin{itemize}
\item[$\bullet$] $t\geq s$;
\item[$\bullet$] $\{x_1, \ldots, x_t\}$ is an independent subset of
     vertices of $G$; and
\item[$\bullet$] $\{x_i, x_{t+j}\}\in E(G)$ implies that either $i=j$ or $i\leq j-s$.
\end{itemize}

Set
\begin{align*}
M:=& \Big\{\{x_{i,t+2-s-i}, x_{t+i, k+s+i-t-1}\}\mid 1\leq i\leq t-s\Big\}\cup\\ & \Big\{\{x_{i,1}, x_{t+i, k}\}\mid t-s+1\leq i\leq t\Big\}.
\end{align*}
We show that $M$ is an induced matching of $G_k$. Fix integers $i$ and $j$ with $1\leq i <j\leq t$. Consider the following cases.

\vspace{0.3cm}
{\bf Case 1.}  Suppose $1\leq i,j\leq t-s$. Since the vertices $x_i$ and $x_j$ are not adjacent in $G$, it follows from the construction of $G_k$ that $\{x_{i,t+2-s-i}, x_{j,t+2-s-j}\}\notin E(G_k)$. Since $i<j$, we conclude that $\{x_j, x_{t+i}\}\notin E(G)$ and it again follows from the construction of $G_k$ that $\{x_{j,t+2-s-j}, x_{t+i, k+s+i-t-1}\}$ is not an edge of $G_k$.

If $j<i+s$, then $x_i$ and $x_{t+j}$ are not adjacent in $G$. Therefore, the construction of $G_k$ implies that $\{x_{i,t+2-s-i}, x_{t+j, k+s+j-t-1}\}$ is not an edge of $G_k$. If $j\geq i+s$, then$$(t+2-s-i)+(k+s+j-t-1)=k+j-i+1\geq k+s+1> k+1.$$Hence, $\{x_{i,t+2-s-i}, x_{t+j, k+s+j-t-1}\}$ is not an edge of $G_k$.

Since $k\geq 2t-2s+2$, $i\geq 1$ and $j\geq 2$, we have$$(k+s+i-t-1)+(k+s+j-t-1)=k+(k-2t+2s-2)+i+j> k+2.$$Thus, $\{x_{t+i,k+s+i-t-1}, x_{t+j,k+s+j-t-1}\}$ is not an edge of $G_k$.

\vspace{0.3cm}
{\bf Case 2.} Suppose $1\leq i\leq t-s$ and $j\geq t-s+1$. Since the vertices $x_i$ and $x_j$ are not adjacent in $G$, we have $\{x_{i,t+2-s-i}, x_{j,1}\}\notin E(G_k)$. Since, $i<j$, we conclude that $\{x_j, x_{t+i}\}\notin E(G)$ and so, $\{x_{j,1}, x_{t+i, k+s+i-t-1}\}$ is not an edge of $G_k$.

If $j<i+s$, then $x_i$ and $x_{t+j}$ are not adjacent in $G$. Therefore, the construction of $G_k$ implies that $\{x_{i,t+2-s-i}, x_{t+j, k}\}$ is not an edge of $G_k$. If $j\geq i+s$, then it follows from $i\leq t-s$ that$$(t+2-s-i)+k\geq k+2.$$Hence, $\{x_{i,t+2-s-i}, x_{t+j, k}\}$ is not an edge of $G_k$.

Since $k\geq 2t-2s+2$ and $i\geq 1$, we have
\begin{align*}
& k+(k+s+i-t-1)\geq k+(2t-2s+2+s+1-t-1)\\ & =k+t-s+2\geq k+2.
\end{align*}
Thus, $\{x_{t+i,k+s+i-t-1}, x_{t+j,k}\}$ is not an edge of $G_k$.

\vspace{0.3cm}
{\bf Case 3.} Suppose $i,j\geq t-s+1$. Since the vertices $x_i$ and $x_j$ are not adjacent in $G$, we have $\{x_{i,1}, x_{j,1}\}\notin E(G_k)$. Since, $i<j$, we conclude that $\{x_j, x_{t+i}\}\notin E(G)$ and so, $\{x_{j,1}, x_{t+i,k}\}$ is not an edge of $G_k$.

The inequalities $i\geq t-s+1$ and $j\leq t$ imply that $j<i+s$. Therefore, $x_i$ and $x_{t+j}$ are not adjacent in $G$. Hence, $\{x_{i,1}, x_{t+j,k}\}$ is not an edge of $G_k$. Since, $k\geq 2$, we have $2k\geq k+2$. In particular,  $\{x_{t+i,k}, x_{t+j,k}\}$ is not an edge of $G_k$.

\vspace{0.3cm}
It follows from the above cases that $M$ is an induced matching of $G_k$. Thus,$$\ind-match(G_k)\geq t=\ord-match(G)$$and the assertion follows.
\end{proof}

Cook and Nagel \cite{cn} defined the concept of fully clique-whiskered graphs in the following way. For a
given graph $G$, a subset $W\subseteq V(G)$ is called a {\it clique} of $G$ if every
pair of vertices of $W$ are adjacent in $G$. Let $\pi$ be a partition of
$V(G)$, say $V(G)=W_1\cup\cdots \cup W_{m}$, such that $W_i$ is a clique of
$G$ for every $1\leq i \leq m$. Then we say that $\pi=\{W_1, \ldots, W_{m}\}$ is a {\it clique vertex-partition} of $G$. Add new vertices $y_1,\ldots,y_{m}$ and new edges $\{x, y_i\}$ for every $x\in W_i$ and every $1\leq i\leq m$. The resulting
graph is called a {\it fully clique-whiskered graph} of $G$, denoted by $G^{\pi}$. Using Theorem \ref{main}, we are able to determine the depth of symbolic powers of cover ideals of fully clique-whiskered graphs.

\begin{cor} \label{whisk}
Let $G$ be a graph and suppose that $\pi=\{W_1, \ldots, W_{m}\}$ is a clique vertex-partition of $G$. Then
$${\rm depth}(S/J(G^{\pi})^{(k)})=
\left\{
	\begin{array}{ll}
		n+m-\alpha(G)-1  & \mbox{if } k=1 \\
		n-1 & \mbox{if } k\geq 2
	\end{array}
\right.$$
In particular, ${\rm sdstab}(J(G^{\pi}))\leq 2$.
\end{cor}

\begin{proof}
We know from \cite[Theorem 13]{bv} that ${\rm reg}(S/I(G^{\pi}))=\alpha(G)$. Therefore, Terai's theorem \cite[Theorem 8.1.10]{hh} implies that ${\rm pd}(S/J(G^{\pi}))=\alpha(G)+1$. Since $|V(G^{\pi})|=n+m$, the Auslander-Buchsbaum formula implies the assertion for $k=1$.

Next, we consider the case $k\geq 2$. We use the notations introduced just before the corollary. For every integer integer $i$ with $1\leq i\leq m$, choose a vertex $w_i\in W_i$. One can easily see that $\ord-match(G^{\pi})=m$ and$$M=\big\{\{w_i,y_i\}\mid 1\leq i\leq m\big\}$$is an ordered matching of size $m$ in $G^{\pi}$. Since for every pair of integers $i\neq j$, we have $\{w_i,y_j\}\notin E(G^{\pi})$, we deduce that $M$ is a $m$-ordered matching of $G$. In particular, $\mord-match(G^{\pi})=\ord-match(G^{\pi})$. The assertion now follows from Theorem \ref{main} by noticing that $|V(G^{\pi})|=n+m$.
\end{proof}

Katzman \cite{k} proves that for any graph $G$, we have$${\rm reg}(S/I(G))\geq \ind-match(G).$$
It is known that for several classes of graphs, including chordal graphs \cite[Corollary 6.9]{hv}, Sequentially Cohen-Macaulay bipartite graphs \cite[Theorem 3.3]{va} and very well-covered graphs \cite[Theorem 4.12]{mmcrty}, equality holds in the above inequality. In the following corollary, using the proof of Theorem \ref{main}, we provide a new class of graphs $G$ for which we have equality in the above mentioned inequality.

\begin{cor} \label{regind}
Let $G$ be a graph with $n$ vertices and suppose that $s\geq 1$ is the largest integer such that $\sord-match(G)=\ord-match(G)$. Also, let $G_k$ be the graph introduced in Construction \ref{cons}.
\begin{itemize}
\item[(i)] If $s=1$, then for every integer $k\geq 2\ord-match(G)-1$, we have$${\rm reg}(I(G_k))=\ind-match(G_k)+1=\ord-match(G)+1.$$

\item[(ii)] If $s\geq 2$, then for every integer $k\geq 2\ord-match(G)-2s+2$, we have$${\rm reg}(I(G_k))=\ind-match(G_k)+1=\ord-match(G)+1.$$
\end{itemize}
\end{cor}

\begin{proof}
Set $t:=\ord-match(G)$. We first prove (ii). So, suppose $s\geq 2$ and let $k$ be an integer with $k\geq 2\ord-match(G)-2s+2$. We know from Theorem \ref{main} that ${\rm depth}(S/J(G)^{(k)})=n-t-1$. Consequently, the Auslander-Buchsbaum formula implies that ${\rm pd}(S/J(G)^{(k)})=t+1$. As polarization preserves the projective dimension, we have$${\rm pd}(T/(J(G)^{(k)})^{{\rm pol}})=t+1,$$where $T$ is a new polynomial ring containing $(J(G)^{(k)})^{{\rm pol}}$. We know from \cite[Lemma 3.4]{s3} that $(J(G)^{(k)})^{{\rm pol}}$ is the cover ideal of $G_k$. Using Terai's theorem \cite[Theorem 8.1.10]{hh}, we deduce that that ${\rm reg}(I(G_k))= t+1$.

It follows from \cite[Lemma 2.2]{k} that$${\rm reg}((I(G_k))\geq \ind-match(G_k)+1.$$On the other hand in the proof of Theorem \ref{main}, we showed that $\ind-match(G_k)\geq t$. Thus,$$t+1={\rm reg}(I(G_k))\geq \ind-match(G_k)+1\geq t+1$$and the assertion follows.

The proof of (i) is similar to that of (ii). The only difference is that one should use the proof of \cite[Theorem 3.1]{s5} to conclude the inequality $\ind-match(G_k)\geq t$.
\end{proof}

\begin{rem}
It is known that for any graph $G$, the inequality$${\rm reg}(I(G))\leq \ord-match(G)+1$$holds (see for instance \cite[Remark 4.8]{cv} or \cite[Corollary 2.5]{s4}). Corollary \ref{regind} provides an alternative proof for this inequality. Indeed, let $k\gg 0$ be an integer. Note that $G$ is an induced subgraph $G_k$. We know from \cite[Lemma 3.1]{h} that if $H'$ is an induced subgraph of a graph $H$, then ${\rm reg}(I(H'))\leq {\rm reg}(I(H))$. Thus, Corollary \ref{regind} implies that$${\rm reg}(I(G))\leq {\rm reg}(I(G_k))=\ord-match(G)+1.$$
\end{rem}

Hang and Trung \cite{ht} proved that for any bipartite graph $G$, we have ${\rm dstab}(J(G))\leq \ord-match(G)$. In the following theorem, we provide an alternative proof for their result.

\begin{thm} [\cite{ht}, Theorem 3.6] \label{bipart}
Let $G$ be a bipartite graph with $n$ vertices. Then for every integer $k\geq \ord-match(G)$, we have$${\rm depth}(S/J(G)^k)=n-\ord-match(G)-1.$$In particular, ${\rm dstab}(J(G))\leq \ord-match(G)$.
\end{thm}

\begin{proof}
Suppose $V(G)=\{x_1, \ldots, x_n\}$ and assume without loss of generality that $G$ has no isolated vertex. Since $G$ is a bipartite graph, it follows from \cite[Corollary 2.6]{grv} that $J(G)^{(k)}=J(G)^k$, for each integer $k\geq 1$. Thus, combining \cite[Proposition 2.4]{v} and \cite[Theorem 2.8]{cv}, we have$${\rm depth}(S/J(G)^k)={\rm depth}(S/J(G)^{(k)})\geq n-\ord-match(G)-1,$$for every $k\geq 1$. Hence, we only need to prove that$${\rm depth}(S/J(G)^k)\leq n-\ord-match(G)-1,$$for every integer $k\geq \ord-match(G)$. According to  Auslander-Buchsbaum formula, we need to show that ${\rm pd}(S/J(G)^k)\geq  \ord-match(G)+1$. Since polarization does not change the projective dimension, we prove that$${\rm pd}(T/(J(G)^k)^{{\rm pol}})\geq \ord-match(G)+1,$$where $T$ is a new polynomial ring containing $(J(G)^k)^{{\rm pol}}$. We know from \cite[Lemma 3.4]{s3} that $(J(G)^k)^{{\rm pol}}=(J(G)^{(k)})^{{\rm pol}}$ is the cover ideal of the graph $G_k$. Using Terai's theorem \cite[Theorem 8.1.10]{hh}, it is sufficient to show that $${\rm reg}(T/(I(G_k))\geq \ord-match(G).$$By \cite[Lemma 2.2]{k}, we know that$${\rm reg}(T/(I(G_k))\geq \ind-match(G_k).$$We prove that $\ind-match(G_k)\geq \ord-match(G)$ and this completes the proof of the theorem. Set $t=\ord-match(G)$ and assume without loss of generality that$$\big\{\{x_i,x_{t+i}\}\mid 1\leq i\leq t\big\}$$is an ordered matching of $G$ so that
\begin{itemize}
\item[$\bullet$] $\{x_1, \ldots, x_t\}$ is an independent subset of
     vertices of $G$.
\item[$\bullet$] $\{x_i, x_{t+j}\}\in E(G)$ implies that $i\leq j$.
\end{itemize}
Moreover, using \cite[Lemma 3.4]{ht}, we may assume that $\{x_{t+1}, x_{t+2} \ldots, x_{2t}\}$ is an independent subset of vertices of $G$.

Consider the set of edges$$M=\big\{\{x_{i,t+1-i}, x_{t+i, k+i-t}\}\mid 1\leq i\leq t\big\}.$$Note that since $k\geq t$, for each integer $i$ with $1\leq i\leq t$, we have$$1\leq t+1-i, k+i-t\leq k.$$Therefore, $x_{i,t+1-i}$ and $x_{t+i, k+i-t}$ are vertices of $G_k$.

We show that $M$ is an induced matching of $G_k$.   Fix integers $i$ and $j$ with $1\leq i <j\leq t$. As $\{x_1, \ldots, x_t\}$ is an independent subset of
vertices of $G$, the vertices $x_{i,t+1-i}$ and $x_{j,t+1-j}$ are not adjacent in $G_k$. Similarly, $x_{t+i,k+i-t}$ and $x_{t+j,k+j-t}$ are not adjacent in $G_k$. Since, $i<j$, we have $\{x_j, x_{t+i}\}\notin E(G)$ and consequently, $\{x_{j,t+1-j}, x_{t+i, k+i-t}\}\notin E(G_k)$. On the other hand,$$(t+1-i)+(k+j-t)=k+j-i+1>k+1,$$which implies that $\{x_{i,t+1-i}, x_{t+j, k+j-t}\}$ is not an edge of $G_k$.

Hence, $M$ is an induced matching of $G_k$ which yields the inequality$$\ind-match(G_k)\geq t=\ord-match(G),$$and this completes the proof.
\end{proof}

\section*{Acknowledgment}

The first author is supported by the FAPA grant from the Universidad de los Andes.



\end{document}